\documentclass[12pt,a4paper]{article}

\usepackage{graphicx,tikz}
\usepackage{latexsym}
\usepackage{amsmath}
\usepackage{amssymb}
\usepackage{arydshln}

\usepackage{cite}
\usepackage[comma]{natbib}
\setcitestyle{numbers,square}
\usepackage{stmaryrd}
\usepackage{enumerate}

\usepackage{hyperref}

\usepackage{color}
\usepackage{lineno}
\usepackage{graphicx}
\usepackage{ae}
\usepackage{amsmath}
\usepackage{amssymb}
\usepackage{latexsym}
\usepackage{url}
\usepackage{epsfig}
\usepackage{mathrsfs}
\usepackage{amsfonts}
\usepackage{amsthm}
\usepackage{bm}

\newcommand{\Cay}{\mathrm{Cay}}

\newcommand{\Spec}{\mathrm{Spec}}

\newtheorem{theorem}{Theorem}[section]

\newtheorem{lemma}[theorem]{Lemma}
\newtheorem{remark}{Remark}
\newtheorem{cor}[theorem]{Corollary}

\theoremstyle{definition}

\newtheorem{problem}{Problem}

\numberwithin{equation}{section} 
\allowdisplaybreaks    

\def\qed{\hfill$\Box$\vspace{12pt}}

\long\def\delete#1{}

\usepackage{xcolor}
\usepackage[normalem]{ulem}

\usepackage{stfloats}

\begin{document}
\title{Integral Cayley graphs over a group of order $6n$}
\author{Jing Wang$^{a,b,c}$,~Xiaogang Liu$^{a,b,c}$$^,$\thanks{Supported by the National Natural Science Foundation of China (No. 12371358) and the Guangdong Basic and Applied Basic Research Foundation (No. 2023A1515010986).}~, ~Ligong Wang$^{a,b,c}$$^,$\thanks{Supported by the National Natural Science Foundation of China (No. 12271439).}~$^,$\thanks{Corresponding author. Email addresses: wj66@mail.nwpu.edu.cn, lgwangmath@163.com, xiaogliu@nwpu.edu.cn}
\\[2mm]
{\small $^a$School of Mathematics and Statistics,}\\[-0.8ex]
{\small Northwestern Polytechnical University, Xi'an, Shaanxi 710072, P.R.~China}\\
{\small $^b$Research \& Development Institute of Northwestern Polytechnical University in Shenzhen,}\\[-0.8ex]
{\small Shenzhen, Guandong 518063, P.R. China}\\
{\small $^c$Xi'an-Budapest Joint Research Center for Combinatorics,}\\[-0.8ex]
{\small Northwestern Polytechnical University, Xi'an, Shaanxi 710129, P.R. China}\\
}
\date{}

\date{}

\openup 0.5\jot
\maketitle
\begin{abstract}
In this paper, we study the integral Cayley graphs over a non-abelian group $U_{6n}=\langle a,b\mid a^{2n}=b^3=1, a^{-1}ba=b^{-1}\rangle$ of order $6n$. We give a necessary and sufficient condition for the integrality of Cayley graphs over $U_{6n}$. We also study relationships between the integrality of Cayley graphs over $U_{6n}$ and the Boolean algebra of cyclic groups. As applications, we construct some infinite families of connected integral Cayley graphs over $U_{6n}$.


\end{abstract}

\section{Introduction}
Let $\Gamma$ be a simple graph with vertex set $V(\Gamma)$ and edge set $E(\Gamma)$. The \emph{adjacency matrix} of $\Gamma$ is denoted by $A(\Gamma)=(a_{u,v})_{u,v\in V(\Gamma)}$, where $a_{u,v}=1$ if $uv\in E(\Gamma)$, and $a_{u,v}=0$ otherwise. The set of all eigenvalues of $A(\Gamma)$ is called the \emph{spectrum} $\Spec(\Gamma)$ of the graph $\Gamma$.
If all eigenvalues of $A(\Gamma)$ are integers, then $\Gamma$ is called an \emph{integral graph}. In 1974, Harary and Schwenk \cite{HararyS1974} proposed the following problem:
\begin{problem}\label{integralU6n-problem1}
Which graphs are integral?
\end{problem}
Since then, characterizing integral graphs has become an important research topic in algebraic graph theory. Although many classes of integral graphs have been characterized (see \cite{BussemakerC1976,WatanabeS1979,Wang2005,BalinskaCRSS2002}), it is still very difficult to give a complete solution to Problem 1, which is far from being settled. Recently, it has been found that integral graphs are good candidates to have perfect state transfer \cite{chris1}, which have significant applications in quantum computation theory \cite{ChristandlDDEKL2005, Dutta2023, Kay2010}.

Let $G$ be a group with the identity element $1$ and $S$ a subset of $G$ such that  $1\notin S=S^{-1}=\left\{s^{-1}\mid s\in S\right\}$ (inverse-closed). The \emph{Cayley graph} $\Gamma=\Cay(G,S)$ is a graph whose vertex set is $G$ and edge set is $\{\{g,sg\}\mid g\in G, s\in S\}$.
By definition, $\Cay(G,S)$ is a simple undirected regular graph, which is connected if and only if $S$ generates $G$. 
In 2009, Abdollahi and Vatandoost \cite{AbdollahiV2009} determined all connected cubic integral Cayley graphs and proposed the following problem:
\begin{problem}\label{integralU6n-problem1}
Which Cayley graphs are integral?
\end{problem}
Thereafter, many results on Problem 2 have been given. Klotz and Sander \cite{KlotzS2010,KlotzS2011} determined all finite abelian Cayley integral groups~(A finite group $G$ is said to be \emph{Cayley integral} if every undirected Cayley graph over $G$ is integral).
Alperin and Peterson \cite{AlperinP2012} given a necessary and sufficient condition for the integrity of Cayley graphs over abelian groups.
Lu et al. \cite{LuHH2018} studied the integrality of Cayley graphs over non-abelian groups and obtained a necessary and sufficient condition for the integrality of Cayley graphs over dihedral groups $D_n$. Subsequently, this result was extended to generalized dihedral groups by Huang and Li \cite{HuangL2021}. Some other results on integrality of Cayley graphs over non-abelian groups were also studied. Cheng et al. studied a necessary and sufficient condition for the integrality of Cayley graphs over dicyclic groups $T_{4n}$ \cite {ChengFH2019} and a certain group of order $8n$ \cite{ChengFL2022}. And then Behajaina and Legrand \cite{BehajainaL2022} extended the works on dicyclic groups to generalized dicyclic groups. For more results, the readers can refer to \cite{AbdollahiJ2014,KuLW2014,DeVosKMA2013,EstelyiK2014,Ghasemi2017,MaW2016,So2006}.

Recently, in an excellent survey on eigenvalues of Cayley graphs \cite{LiuZ2022}, Liu and Zhou proposed to study the following problem:

\begin{problem}\label{integralU6n-problem1}
Characterize integral Cayley graphs over the non-abelian group $U_{6n}$ for any $n\geq1$, where
$$U_{6n}=\langle a,b\mid a^{2n}=b^3=1, a^{-1}ba=b^{-1}\rangle.$$
\end{problem}
As far as we know, the group $U_{6n}$ is first discovered in the book \cite{JamesL01} by James and Liebeck. Recently, many properties of Cayley  graphs over $U_{6n}$ have been investigated (see \cite{AssariS2014, DarafshehY2017,DarafshehY2018}). In this paper, we address the Problem \ref{integralU6n-problem1}.

The paper is organized as follows. In Section 2, we introduce some useful lemmas. In Section 3, we first get a necessary and sufficient condition for the integrality of Cayley graphs over $U_{6n}$ by using Babai's result on spectra of Cayley graphs (see Theorem \ref{integralU6n-theorem2}). Next, we give another version of the necessary and sufficient condition  for the integrality of Cayley graphs over $U_{6n}$ which expressed by the irreducible characters of cyclic groups (see Theorem \ref{integralU6n-proposition1}). In Section 4, we study relationships between the integrality of Cayley graphs over $U_{6n}$ and the Boolean algebra of cyclic groups (see Theorems \ref{integralU6n-theorem3}, \ref{integralU6n-theorem4} and \ref{integralU6n-theorem5}). Moreover, we also determine some infinite families of connected integral Cayley graphs over $U_{6n}$ (see Corollaries \ref{integralU6n-cor4}, \ref{integralU6n-cor1}, \ref{integralU6n-cor2} and \ref{integralU6n-cor3}).

\section{Preliminaries}

In this section, we introduce some useful lemmas. For more details on representation theory, see \cite{JamesL01,Steinberg12}.

\begin{lemma}\label{integralU6n-lemma6}\emph{(\cite[Lemma 2.1]{AbdollahiV2009})}
Let $\omega=\exp(\frac{\pi \imath}{n})$, where $\imath^2=-1$. Then
\begin{itemize}
\item[\rm(1)]$\sum_{r=1}^{2n-1}\omega^r=-1$,
\item[\rm(2)] If $l$ is even, for $0<l<2n-1$, then $\sum_{r=1}^{n-1}\omega^{lr}=-1$.
\end{itemize}
\end{lemma}

\begin{lemma}\emph{(\cite[Theorem 3.1]{Babai1979})} \label{integralU6n-lemma1}
Let $G$ be a finite group, and $\{\chi_1,\ldots, \chi_h\}$ the set of all irreducible characters over complex field  $\mathbb{C}$ with $\chi_i(1)=d_i~(i=1,\ldots, h)$. Then the spectrum of the Cayley graph $\Cay(G, S)$ can be arranged as
\begin{equation*}
\Spec(\Cay(G, S))=\left\{[\lambda_{11}]^{d_1},\ldots, [\lambda_{1d_1}]^{d_1},\ldots,[\lambda_{h1}]^{d_h},\ldots, [\lambda_{hd_h}]^{d_h}   \right\},
\end{equation*}
where $[\lambda_{ij}]^{d_i}$ denotes that the eigenvalue $\lambda_{ij}$ has the multiplicity $d_i$ corresponding to the character $\chi_i$ for $1\leq j\leq d_i$. Furthermore, for any natural number $t$, we have
\begin{equation*}
\lambda^t_{i1}+\lambda^t_{i2}+\cdots+\lambda^t_{id_i}=\sum\limits_{s_1,\ldots,s_t\in S}\chi_i\left(\Pi_{l=1}^t s_l\right).
\end{equation*}
\end{lemma}

The following results on the group $U_{6n}$ can be verified easily.
\begin{lemma}\label{integralU6n-lemma3}
For the group $U_{6n}$, we have
\begin{itemize}
\item[\rm(1)] $ba^{2r}=a^{2r}b$, $ba^{2r+1}=a^{2r+1}b^{-1}$.
\item[\rm(2)] $(a^{2r}b)^{-1}=a^{-2r}b^2$, $(a^{2r+1}b)^{-1}=a^{-(2r+1)}b$, $(a^{2r+1}b^2)^{-1}=a^{-(2r+1)}b^2$.
\end{itemize}
\end{lemma}

For $0\leq r \leq n-1$, it is known \cite{JamesL01} that the $3n$ conjugacy classes of $U_{6n}$ are
$$\{a^{2r}\},~~~~ \{a^{2r}b, a^{2r}b^2\},~~~~ \{a^{2r+1}, a^{2r+1}b, a^{2r+1}b^2\}.$$
Note that each element in the same conjugacy class has the same character. Then all characters of $U_{6n}$ are given in the following lemma.

\begin{lemma}\emph{(\cite[ Exercise 18.4]{JamesL01})}\label{integralU6n-lemma2}
Let $\omega$ be as in Lemma \ref{integralU6n-lemma6}. The character table of $U_{6n}$ is given in Table \ref{integralU6n-table1}, where the ${\chi_j}'s$ are $2n$ linear characters for $0\leq j \leq 2n-1$ and the ${\psi_k}'s$ are $n$ irreducible characters of degree two for $0\leq k\leq n-1$.
\end{lemma}

\begin{table}[t]
\begin{center}
\caption{Character Table of $U_{6 n}$.}\label{integralU6n-table1}
\begin{tabular}{l|lll}
\hline~~~~~~                                ~~~~& $a^{2r}$            &~~~~ $a^{2r}b$             &~~~~ $a^{2r+1}$\\[0.2mm]
\hline\\[-0.3cm]~~~~~~$\chi_j$~($0\leq j\leq2n-1$)                             ~~~~& $\omega^{2jr}$      &~~~~ $\omega^{2jr}$        &~~~~ $\omega^{j(2r+1)}$\\[0.3cm]

      ~~~~~~$\psi_k$ ($0\leq k\leq n-1$)                          ~~~~& $2\omega^{2kr}$     &~~~~ $-\omega^{2kr}$       &~~~~ $0$\\

\hline
\end{tabular}
\end{center}
\end{table}

By Lemmas \ref{integralU6n-lemma1} and \ref{integralU6n-lemma2}, we get the following result immediately.
\begin{lemma}
Let $S\subseteq U_{6n}\setminus\{1\}$ such that $S=S^{-1}$. Then
\begin{equation*}
\Spec(\Cay(U_{6n}, S))=\left\{[\lambda_{j}]^{1},[\mu_{k1}]^{2},[\mu_{k2}]^{2}\mid0\leq j\leq2n-1, 0\leq k\leq n-1 \right\},
\end{equation*}
where
\begin{equation}\label{integralU6n-equation1}
\left\{
\begin{array}{ll}
\lambda_j=\sum\limits_{s\in S}\chi_j(s), &0\leq j\leq2n-1;\\[0.4cm]
\mu_{k1}+\mu_{k2}=\sum\limits_{s\in S}\psi_k(s), &0\leq k\leq n-1;\\[0.4cm]
\mu^2_{k1}+\mu^2_{k2}=\sum\limits_{s,t\in S}\psi_k(st), &0\leq k\leq n-1.
\end{array}
\right.
\end{equation}
\end{lemma}

The following result gives the characters of cyclic groups.
\begin{lemma}\emph{(\cite[Example 4.4.10]{Steinberg12})}\label{integralU6n-lemma5}
Let $\omega$ be as in Lemma \ref{integralU6n-lemma6}.
If $\langle a \rangle$ is a cyclic group of order $2n$, then the set of all irreducible characters is $\{\rho_j\mid 0\leq j\leq 2n-1\}$, where $\rho_j(a^r)=\omega^{jr},~ 0\leq j, r\leq 2n-1$.
Moreover, the set of all irreducible characters of the subgroup $\langle a^2 \rangle$ is $\{\rho_k\mid 0\leq k\leq n-1\}$, where $\rho_k(a^{2r})=\omega^{2kr},~0\leq k, r\leq n-1$.
\end{lemma}

Let $G$ be a finite group and $\mathcal{F}_G$ the set of all subgroups of $G$. The \emph{Boolean algebra} $B(G)$ is the set whose elements are obtained by arbitrarily finite intersections, unions, and complements of the elements in $\mathcal{F}_G$.
The subset $S\subseteq G$ is called an \emph{integral set} if $\chi(S)=\sum\limits_{s\in S}\chi(s)$ is an integer for each character $\chi$ of $G$. The following lemma illustrates the relationship between integral set and Boolean algebra for an abelian group.
\begin{lemma}\emph{(\cite{AlperinP2012})} \label{integralU6n-lemma4}
Let $G$ be an abelian group. Then $S\subseteq G$ is an integral set if and only if $S\in B(G)$ if and only if $\Cay(G,S)$ is integral.
\end{lemma}

\section{The necessary and sufficient conditions for the integrality of $\Cay(U_{6n},S)$}
In this section, we study the necessary and sufficient conditions for the integrality of $\Cay(U_{6n},S)$.  Let $A,B$ be two subsets of a group $G$. For any character $\chi$ of $G$, define $$\chi(A)=\sum\limits_{a\in A}\chi(a) \text{~~and~~} \chi(AB)=\sum\limits_{a\in A, b\in B}\chi(ab).$$ In particular, $$\chi(\emptyset)=0 \text{~~and~~}
 \chi(A^2)=\sum\limits_{a_1,a_2\in A}\chi(a_1a_2).$$

\begin{theorem}\label{integralU6n-theorem2}
Let $S=S_1\cup S_2\subseteq U_{6n}\setminus\{1\}$  such that $1\notin S=S^{-1}$, where $S_1\subseteq\langle a^2\rangle\cup\langle a^2\rangle b\cup\langle a^2\rangle b^2$ and $S_2\subseteq\langle a^2\rangle a\cup\langle a^2\rangle ab\cup\langle a^2\rangle ab^2$. Then $\Cay(U_{6n},S)$ is integral if and only if the following conditions hold for $0\leq j\leq2n-1$ and $0\leq k\leq n-1$:
\begin{itemize}
\item[\rm(1)] $\chi_j(S)$, $\psi_k(S_1)$ and $\psi_k(S^2_1)+\psi_k(S^2_2)$ are integers;
\item[\rm(2)] $\Delta_{\psi_k}(S)=2\left(\psi_k(S_1^2)+\psi_k(S_2^2)\right)-\psi_k^2(S_1)$ is a perfect square.
\end{itemize}
\end{theorem}

\begin{proof}
By calculation, we have $S_1S_2=\{s_1s_2\mid s_1\in S_1, s_2\in S_2\}\subseteq\langle a^2\rangle a\cup\langle a^2\rangle ab\cup\langle a^2\rangle ab^2$ and $S_2S_1=\{s_2s_1 \mid  s_2\in S_2, s_1\in S_1\}\subseteq\langle a^2\rangle a\cup\langle a^2\rangle ab\cup\langle a^2\rangle ab^2$. By Lemma \ref{integralU6n-lemma2}, $\psi_k(S_1S_2)=\psi_k(S_2S_1)=0$. Thus,
\begin{align*}
\psi_k(S)&=\sum\limits_{s_1\in S_1}\psi_k(s_1)+\sum\limits_{s_2\in S_2}\psi_k(s_2)=\psi_k(S_1),~ 0\leq k\leq n-1,\\[0.2cm]
\psi_k(S^2)&=\sum\limits_{s,t\in S}\psi_k(st)=\psi_k(S_1^2)+\psi_k(S_1S_2)+\psi_k(S_2S_1)+\psi_k(S_2^2)\\
&=\psi_k(S_1^2)+\psi_k(S_2^2),~ 0\leq k\leq n-1.
\end{align*}
So we simplify Equation (\ref{integralU6n-equation1}) as follows:
\begin{equation}\label{integralU6n-equation2}
\left\{
\begin{array}{ll}
 \lambda_j= \chi_j(S), &0\leq j\leq2n-1;\\[0.2cm]
 \mu_{k1}+\mu_{k2}=\psi_k(S_1), &0\leq k\leq n-1;\\[0.2cm]
 \mu^2_{k1}+\mu^2_{k2}=\psi_k(S_1^2)+\psi_k(S_2^2), &0\leq k\leq n-1.
\end{array}
\right.
\end{equation}
Suppose that $\Cay(U_{6n},S)$ is integral. By (\ref{integralU6n-equation2}), $\chi_j(S)$, $\psi_k(S_1)$ and $\psi_k(S_1^2)+\psi_k(S_2^2)$ must be integers, yielding (1). Note that integers $\mu_{k1}$ and $\mu_{k2}$ are roots of the following equation:
\begin{equation}\label{integralU6n-equation3}
x^2-\psi_k(S_1)x+\frac{1}{2}\left(\psi^2_k(S_1)- \left(\psi_k(S_1^2)+\psi_k(S_2^2)\right)\right)=0.
\end{equation}
Then the discriminant $\Delta_{\psi_k}(S)=2\left(\psi_k(S_1^2)+\psi_k(S_2^2)\right)-\psi_k^2(S_1)$ must be a perfect square, and then (2) holds.

Conversely, suppose that (1) and (2) hold. Then $\lambda_j$ must be an integer for all $0\leq j\leq 2n-1$. And the solutions $\mu_{k1}$ and $\mu_{k2}$ of (\ref{integralU6n-equation3}) must be rational. Since $\mu_{k1}$ and $\mu_{k2}$ are algebraic integers, they must be integers. Thus $\Cay(U_{6n},S)$ is integral.
\qed
\end{proof}

Let $S_1\subseteq\langle a^2\rangle\cup\langle a^2\rangle b\cup\langle a^2\rangle b^2$ be a subset of $U_{6n}$ such that $S_1=S_1^{-1}$. By Lemma \ref{integralU6n-lemma3}, $S_1$ can be expressed as follows:
$$S_1=\{a^{2r}\mid r\in R\}\cup \{a^{2l}b, a^{-2l}b^2 \mid l\in L\},$$
where $R,L\subseteq\{0,1,\ldots, n-1\}$ and $S_R:=\{a^{2r}\mid r\in R\}=S_R^{-1}$.

Recall Lemma \ref{integralU6n-lemma5} that the set of all irreducible characters of a cyclic group $\langle a \rangle$ of order $2n$ is $\{\rho_j\mid 0\leq j\leq 2n-1\}$, where $\rho_j(a^{r})=\omega^{jr}, ~ 0\leq j, r\leq 2n-1$. The following result gives another version of the necessary and sufficient conditions for the integrality of $\Cay(U_{6n},S)$.


\begin{theorem}\label{integralU6n-proposition1}
Let $S_1=\{a^{2r}\mid r\in R\}\cup \{a^{2l}b, a^{-2l}b^2\mid l\in L\}$, where $R,L\subseteq\{0,1,\ldots, n-1\}$, and $S_2\subseteq\langle a^2\rangle a\cup\langle a^2\rangle ab\cup\langle a^2\rangle ab^2$ such that $S=S_1\cup S_2$ satisfies $1\notin S=S^{-1}$. Set $S_R=\{a^{2r}\mid  r\in R\}$ and $S_L=\{a^{2l}\mid  l\in L\}$. If $S_L=S_L^{-1}$, then $\Cay(U_{6n},S)$ is integral if and only if the following conditions hold for $0\leq j\leq 2n-1$ and $0\leq k\leq n-1$:
\begin{itemize}
\item[\rm(1)] $3\rho_j(S_R)+\chi_j(S_2)$ and $3\rho_j(S_L)+\chi_j(S_2)$ are integers;
\item[\rm(2)] $2\psi_k(S_2^2)$ is a perfect square,
\end{itemize}
where $\{\rho_j\mid 0\leq j\leq 2n-1\}$ is the set of all irreducible characters of $\langle a \rangle$ of order $2n$.
\end{theorem}

\begin{proof}
Firstly, by Lemma \ref{integralU6n-lemma2}, we have
\begin{align}\label{integralU6n-equation4}\nonumber
\psi_k(S_1)&=\sum\limits_{r\in R}\psi_k(a^{2r})+\sum\limits_{l\in L}\left(\psi_k(a^{2l}b)+\psi_k(a^{-2l}b^2)\right)\\ \nonumber
&=\sum\limits_{r\in R}\psi_k(a^{2r})+2\sum\limits_{l\in L}\psi_k(a^{2l}b)\\ \nonumber
&=2\sum\limits_{r\in R}\omega^{2kr}-2\sum\limits_{l\in L}\omega^{2kl}\\
&=2(\rho_k(S_R)-\rho_k(S_L)).
\end{align}
Recall that $ba^{2r}=a^{2r}b$ for all $0\leq r\leq n-1$ and $S_L=S_L^{-1}$. Then
\begin{align}\label{integralU6n-equation6}\nonumber
\psi_k(S_1^2)&=\sum_{r_1, r_2\in R}\psi_k(a^{2(r_1+r_2)})+2\sum_{r\in R, l\in L}\psi_k(a^{2(r+l)}b)+2\sum_{r\in R, l\in L}\psi_k(a^{2(r+l)}b^2)\\ \nonumber
&+2\sum_{l_1, l_2\in L}\psi_k(a^{2(l_1+l_2)})+\sum_{l_1, l_2\in L}\psi_k(a^{2(l_1+l_2)}b^2)+\sum_{l_1, l_2\in L}\psi_k(a^{2(l_1+l_2)}b)\\ \nonumber
&=2\sum_{r_1, r_2\in R}\omega^{2k(r_1+r_2)}-4\sum_{r\in R,l\in L}\omega^{2k(r+l)}+2\sum_{l_1,l_2\in L}\omega^{2k(l_1+l_2)}\\ \nonumber
&=2\left(\sum\limits_{r_1\in R}\omega^{2kr_1}\sum\limits_{r_2\in R}\omega^{2kr_2}-2\sum\limits_{r\in R}\omega^{2kr}\sum\limits_{l\in L}\omega^{2kl}+\sum\limits_{l_1\in L}\omega^{2kl_1}\sum\limits_{l_2\in L}\omega^{2kl_2}\right)\\ \nonumber
&=2\left(\rho_k^2(S_R)-2\rho_k(S_R)\rho_k(S_L)+\rho_k^2(S_L)\right)\\
&=2(\rho_k(S_R)-\rho_k(S_L))^2.
\end{align}
If $\Cay(U_{6n},S)$ is integral, by Theorem \ref{integralU6n-theorem2} (1) and Equation (\ref{integralU6n-equation4}), then both $\psi_k(S_1)=2(\rho_k(S_R)-\rho_k(S_L)),~ 0\leq k\leq n-1$, and
\begin{align}\label{integralU6n-equation5}\nonumber
\chi_j(S)&=\sum\limits_{r\in R}\omega^{2jr}+2\sum\limits_{l\in L}\omega^{2jl}+\chi_j(S_2)\\
&=\rho_j(S_R)+2\rho_j(S_L)+\chi_j(S_2),~ 0\leq j\leq2n-1,
\end{align}
are integers.
Since $\rho_k(S_R)-\rho_k(S_L)$ is an algebraic integer, $\rho_k(S_R)-\rho_k(S_L)$ is an integer for all $0\leq k\leq n-1$. Note that
$$\rho_{n+k}(S_R)-\rho_{n+k}(S_L)=\rho_k(S_R)-\rho_k(S_L),~ 0\leq k\leq n-1,$$
thus $\rho_j(S_R)-\rho_j(S_L)$ is an integer for all $0\leq j\leq 2n-1$. Combining with (\ref{integralU6n-equation5}), we have that $3\rho_j(S_R)+\chi_j(S_2)$ and $3\rho_j(S_L)+\chi_j(S_2)$ are integers, yielding (1).

By Theorem \ref{integralU6n-theorem2} (2) and Equations (\ref{integralU6n-equation4}) and (\ref{integralU6n-equation6}), we have
\begin{align*}
\Delta_{\psi_k}(S)&=2\left(\psi_k(S_1^2)+\psi_k(S_2^2)\right)-\psi_k^2(S_1)\\
&=2\left(2\left(\rho_k(S_R)-\rho_k(S_L)\right)^2+\psi_k(S_2^2)\right)-4(\rho_k(S_R)-\rho_k(S_L))^2\\
&=2\psi_k(S_2^2)
\end{align*}
is a perfect square for all $0\leq k\leq n-1$, and then (2) holds.

Conversely, suppose that (1) and (2) hold. Obviously, Theorem \ref{integralU6n-theorem2} (2) holds. Since $3\rho_j(S_R)+\chi_j(S_2)$ and $3\rho_j(S_L)+\chi_j(S_2)$ are integers, we have $\rho_j(S_R)-\rho_j(S_L)$ is an integer for all $0\leq j\leq 2n-1$. By (\ref{integralU6n-equation4}) and (\ref{integralU6n-equation6}), $\psi_k(S_1)$ and $\psi_k(S_1^2)$ are integers for all $0\leq k\leq n-1$. And by  (\ref{integralU6n-equation5}),
$$\chi_j(S)=3\rho_j(S_R)+\chi_j(S_2)-2\left(\rho_j(S_R)-\rho_j(S_L)\right)$$
is also an integer for all $0\leq j \leq 2n-1$. Since $2\psi_k(S_2^2)$ is a perfect square and $\psi_k(S_2^2)$ is an algebraic integer, $\psi_k(S_2^2)$ must be an integer for all $0\leq k \leq n-1$. Hence, $\psi_k(S_1^2)+\psi_k(S_2^2)$ is an integer for all $0\leq k\leq n-1$, and then Theorem \ref{integralU6n-theorem2} (1) holds. By Theorem \ref{integralU6n-theorem2}, $\Cay(U_{6n}, S)$ is integral.
\qed\end{proof}

\begin{cor}\label{integralU6n-cor4}
Let $n=2p$, where $p$ is a positive integer with $p>2$. Let $S=S_1\cup S_2$ with $S_1=\langle a^4\rangle\setminus \{1\}\cup\{a^{2l}b, a^{-2l}b^2\mid 1\leq l\leq n-1\} $ and $S_2=\{a^{2r+1}b\mid 0\leq r\leq n-1\}$.  Then $\Cay(U_{6n},S)$ is a  connected integral graph whose spectrum is $\{[-3]^{4p-4},[0]^{8n-8},[p]^6, [p-3]^2,3p-3,[-3p]^2,7p-3\}$.
\end{cor}

\begin{proof}
It is evident that $S=S^{-1}$ generates $U_{6n}$, so $\Cay(U_{6n},S)$ is connected. Set $S_R=\langle a^4\rangle\setminus \{1\}=\{a^{4r} \mid 1\leq r\leq p-1\}$ and $S_L=\langle a^2\rangle\setminus \{1\}=\{a^{2l} \mid 1\leq l\leq 2p-1\}$.
By Lemma \ref{integralU6n-lemma6}, we have
\begin{align*}
\rho_j(S_R)=\sum\limits_{r=1}^{p-1}\omega^{4jr}=\sum\limits_{r=1}^{p-1}\exp\left(\frac{\pi \imath}{p}2jr\right)=\left\{
\begin{array}{ll}
p-1, &\text{if}~ j=0,p,2p,3p;\\[0.2cm]
-1, &\text{otherwise},
\end{array}\right.
\end{align*}
\begin{align*}
\rho_j(S_L)=\sum\limits_{l=1}^{2p-1}\omega^{2jl}=\left\{
\begin{array}{ll}
2p-1, & j=0,2p;\\[0.2cm]
-1, & 0<j\leq4p-1,j\neq 2p,
\end{array}\right.
\end{align*}
and
\begin{align*}
\chi_j(S_2)=\sum\limits_{r=0}^{2p-1}\omega^{j(2r+1)}=\omega^j\sum\limits_{r=0}^{2p-1}\omega^{2jr}=\left\{
\begin{array}{ll}
2p,  & j=0;\\[0.2cm]
-2p, & j=2p;\\[0.2cm]
0,   &1\leq j\leq 4p-1, j\neq 2p.
\end{array}\right.
\end{align*}
Obviously, $3\rho_j(S_R)+\chi_j(S_2)$ and $3\rho_j(S_L)+\chi_j(S_2)$ are integers. And
\begin{align*}
2\psi_k(S_2^2)=4p\sum\limits_{r=0}^{2p-1}\psi_k(a^{2r})=8p\sum\limits_{r=0}^{2p-1}\omega^{2kr}=\left\{
\begin{array}{ll}
16p^2, &k=0;\\[0.2cm]
0, & 1\leq k\leq 2p-1,
\end{array}\right.
\end{align*}
is a perfect square. By Theorem \ref{integralU6n-proposition1}, $\Cay(U_{6n},S)$ is integral.

By Equations (\ref{integralU6n-equation2}), (\ref{integralU6n-equation4}), (\ref{integralU6n-equation6}) and (\ref{integralU6n-equation5}), the spectrum of $\Cay(U_{6n}, S)$ is $\{[-3]^{4p-4},$ $[0]^{8p-8},[p]^6, [p-3]^2,3p-3,[-3p]^2,7p-3\}$.
\qed\end{proof}

\begin{cor}
For each positive integer $p>2$, there is a connected $(7p-3)$-regular integral graph with $12p$ vertices.
\end{cor}

\section{Relationships between the integrality of $\Cay(U_{6n}, S)$ and the Boolean algebra of cyclic groups}
In this section, we study relationships between the integrality of $\Cay(U_{6n}, S)$ and the Boolean algebra of cyclic groups.
Based on Theorem \ref{integralU6n-proposition1}, we have the following results.
\begin{theorem}\label{integralU6n-theorem3}
Let $S_1\subseteq\langle a^2\rangle\setminus\{1\}$ and $S_2\subseteq\langle a^2\rangle a\cup\langle a^2\rangle ab\cup\langle a^2\rangle ab^2$ such that $S=S_1\cup S_2$ satisfies $1\notin S=S^{-1}$. Then $\Cay(U_{6n}, S)$ is integral if and only if the following conditions hold for $0\leq j\leq 2n-1$ and $0\leq k\leq n-1$:
\begin{itemize}
\item[\rm(1)] $S_1\in B(\langle a^2\rangle)$;
\item[\rm(2)] $\chi_j(S_2)$ is an integer;
\item[\rm(3)] $2\psi_k(S_2^2)$ is a perfect square.
\end{itemize}
\end{theorem}

\begin{proof}
Recall Lemma \ref{integralU6n-lemma5} that the set of all irreducible characters of $\langle a^2 \rangle$ is $\{\rho_j\mid 0\leq j\leq n-1\}$, where $\rho_j(a^{2r})=\omega^{2jr},~0\leq j, r\leq n-1$. Note that $\rho_{n+j}(S_1)=\rho_j(S_1)$ for $0\leq j\leq n-1$. Then $S_1$ is an integral set if and only if  $\rho_j(S_1)$ is an integer for all $0\leq j\leq n-1$ if and only if $\rho_j(S_1)$ is an integer for all $0\leq j \leq 2n-1$. Notice that $\rho_j(S_1)$ is an algebraic integer for $0\leq j \leq 2n-1$. Then $S_1$ is an integral set if and only if $3\rho_j(S_1)$ is an integer for all $0\leq j \leq 2n-1$. Set $S_R=S_1$ and $S_L=\emptyset$. The required result follows from Lemma \ref{integralU6n-lemma4} and Theorem \ref{integralU6n-proposition1}.
\qed\end{proof}


\begin{cor}\label{integralU6n-cor1}
Let $S=S_1\cup S_2$ with $S_1=\langle a^2\rangle\setminus\{1\}$ and $S_2=\{a^{2r+1}, a^{2r+1}b\mid 0\leq r\leq n-1\}$, where $n>1$ is an integer. Then $\Cay(U_{6n},S)$ is a connected integral graph whose spectrum is $\{-n-1, [-1]^{6n-4}, [2n-1]^2, 3n-1\}$.
\end{cor}

\begin{proof}
It is evident that $S=S^{-1}$ generates $U_{6n}$, so $\Cay(U_{6n},S)$ is connected.  By Lemma \ref{integralU6n-lemma6}, we have
\begin{align*}
\chi_j(S_2)=2\sum\limits_{r=0}^{n-1}\omega^{j(2r+1)}=2\omega^j\sum\limits_{r=0}^{n-1}\omega^{2jr}=\left\{
\begin{array}{ll}
2n, &j=0;\\[0.2cm]
-2n, & j=n;\\[0.2cm]
0, & 1\leq j\leq 2n-1, j\neq n,
\end{array}\right.
\end{align*}
is an integer. And
\begin{align*}
2\psi_k(S_2^2)&=4n\sum\limits_{r=0}^{n-1}\psi_k(a^{2r})+2n\sum\limits_{r=0}^{n-1}\psi_k(a^{2r}b)+2n\sum\limits_{r=0}^{n-1}\psi_k(a^{2r}b^2)\\
&=4n\sum\limits_{r=0}^{n-1}\omega^{2kr}
=\left\{
\begin{array}{ll}
4n^2, &k=0,\\[0.2cm]
0, & 1\leq k\leq n-1,
\end{array}\right.
\end{align*}
is a perfect square. Note that $S_1=\langle a^2\rangle\setminus\{1\}\in B(\langle a^2\rangle)$. Then $\Cay(U_{6n},S)$ is integral by Theorem \ref{integralU6n-theorem3}.

By Equations (\ref{integralU6n-equation2}) and (\ref{integralU6n-equation3}), the spectrum of $\Cay(U_{6n}, S)$ is $\{-n-1, [-1]^{6n-4}, [2n-1]^2, 3n-1\}$.
\qed\end{proof}

\begin{cor}
For each positive integer $n>1$, there is a connected $(3n-1)$-regular integral graph with $6n$ vertices.
\end{cor}

\begin{theorem}\label{integralU6n-theorem4}
Let $S_1=\{a^{2l}b, a^{-2l}b^2 \mid l\in L\}$, where $L\subseteq\{0,1,\ldots, n-1\}$, and $S_2\subseteq\langle a^2\rangle a\cup\langle a^2\rangle ab\cup\langle a^2\rangle ab^2$ such that $S=S_1\cup S_2$ satisfies $1\notin S=S^{-1}$. Set $S_L=\{a^{2l} \mid l\in L\}$. If $S_L=S_L^{-1}$, then $\Cay(U_{6n},S)$ is integral if and only if the following conditions hold for $0\leq k \leq n-1$ and $0\leq j\leq 2n-1$:
\begin{itemize}
\item[\rm(1)]$S_L\in B(\langle a^2\rangle)$;
\item[\rm(2)]$\chi_j(S_2)$ is an integer;
\item[\rm(3)]$2\psi_k(S_2^2)$ is a perfect square.
\end{itemize}
\end{theorem}

\begin{proof}
The proof is similar to that of Theorem \ref{integralU6n-theorem3}. Hence we omit the details here.
\qed\end{proof}

Here we would like to mention that Theorem \ref{integralU6n-theorem4} extends \cite[Theorem 1.4]{AbdollahiV2009}, in which $n>1$ is an odd integer.

\begin{cor}\label{integralU6n-cor2}
Let $S=S_1\cup S_2$ with $S_1=\{a^{2r}b,a^{2r}b^2\mid 1\leq r\leq n-1\}$ and $S_2=\{a^{2r+1}b\mid 0\leq r\leq n-1\}$, where $n>1$ is an integer.  Then $\Cay(U_{6n},S)$ is a connected integral graph whose spectrum is $\{[1-2n]^2,[-2]^{2n-2}, [1]^{4n-2}, n-2, 3n-2\}$.
\end{cor}

\begin{proof}
It is evident that $S=S^{-1}$ generates $U_{6n}$, so $\Cay(U_{6n},S)$ is connected. By Lemma \ref{integralU6n-lemma6}, we have
\begin{align*}
\chi_j(S_2)=\sum\limits_{r=0}^{n-1}\omega^{j(2r+1)}=\left\{
\begin{array}{ll}
n, & j=0;\\[0.2cm]
-n, &j=n;\\[0.2cm]
0, & 1\leq j\leq 2n-1, j\neq n,
\end{array}\right.
\end{align*}
is an integer. And
\begin{align*}
2\psi_k(S_2^2)=2n\sum\limits_{r=0}^{n-1}\psi_k(a^{2r})=
4n\sum\limits_{r=0}^{n-1}\omega^{2kr}=\left\{
\begin{array}{ll}
4n^2, & k=0,\\[0.2cm]
0, &1\leq k\leq n-1,
\end{array}\right.
\end{align*}
is a perfect square. Note that
$$S_L=S_L^{-1}=\{a^{2r}\mid 1\leq r\leq n-1\}=\langle a^2\rangle\setminus\{1\}\in B(\langle a^2\rangle).$$
Then Theorem \ref{integralU6n-theorem4} implies that $\Cay(U_{6n},S)$ is integral. By Equations (\ref{integralU6n-equation2}) and (\ref{integralU6n-equation3}), the spectrum of $\Cay(U_{6n}, S)$ is $\{[1-2n]^2, [-2]^{2n-2}, [1]^{4n-2}, n-2,  3n-2\}$.
\qed\end{proof}

\begin{cor}
For each positive integer $n>1$, there is a connected $(3n-2)$-regular integral graph with $6n$ vertices.
\end{cor}


\begin{theorem}\label{integralU6n-theorem5}
Let $S_1=\{a^{2r}, a^{2r}b, a^{-2r}b^2 \mid r\in R\}$ and $S_2=\{a^{2h+1}, a^{2h+1}b, a^{2h+1}b^2\mid h\in H\}$, where $R,H\subseteq\{0,1,\ldots, n-1\}$, such that $S=S_1\cup S_2$ satisfies $1\notin S=S^{-1}$. Set $S_R=\{a^{2r}\mid r\in R\}$ and $S_H=\{a^{2h+1}\mid h\in H\}$. Then $\Cay(U_{6n}, S)$ is integral if and only if the following conditions hold:
\begin{itemize}
\item[\rm(1)]  $S_R\cup S_H\in B(\langle a\rangle)$;
\item[\rm(2)] $2\psi_k(S_2^2)$ is a perfect square for $0\leq k\leq n-1$.
\end{itemize}
 \end{theorem}

\begin{proof}
By Lemma \ref{integralU6n-lemma2}, we have
$$\chi_j(S_2)=3\sum\limits_{h\in H}\omega^{j(2h+1)},~0\leq j\leq 2n-1.$$
Thus, for $0\leq j\leq 2n-1$,
\begin{align*}
3\rho_j(S_R)+\chi_j(S_2)&=3\left(\sum\limits_{r\in R}\omega^{2jr}+\sum\limits_{h\in H}\omega^{j(2h+1)}\right)
=3\rho_j(S_R\cup S_H).
\end{align*}
Notice that $\rho_j(S_R\cup S_H)$ is an algebraic integer. Then $3\rho_j(S_R\cup S_H)$ is an integer if and only if $\rho_j(S_R\cup S_H)$ is an integer for $0\leq j\leq 2n-1$, i.e., $S_R\cup S_H\subseteq\langle a \rangle$ is an integral set. The required result follows from Lemma \ref{integralU6n-lemma4} and Theorem \ref{integralU6n-proposition1}.
\qed\end{proof}

\begin{cor}\label{integralU6n-cor3}
Let $S=S_1\cup S_2$ with $S_1=\{a^{2r},a^{2r}b,a^{2r}b^2\mid 1\leq r\leq n-1\}$ and $S_2=\{a^{2h+1},a^{2h+1}b,$ $a^{2h+1}b^2\mid 0\leq h\leq n-1\}$, where $n>1$ is an integer. Then $\Cay(U_{6n},S)$ is a connected integral graph whose spectrum is $\{[-3]^{2n-1},[0]^{4n}, 6n-3\}$.
\end{cor}

\begin{proof}
It is evident that $S=S^{-1}$ generates $U_{6n}$, so $\Cay(U_{6n},S)$ is connected.
Set $S_R=S_R^{-1}=\{a^{2r}\mid 1\leq r\leq n-1\}=\langle a^2\rangle\setminus\{1\}$ and $S_H=\{a^{2h+1} \mid 0\leq h\leq n-1\}$. Then $S_R\cup S_H=\langle a \rangle \setminus\{1\}\in B(\langle a \rangle)$.
By calculation, $2\psi_k(S_2^2)=0$ is a perfect square for $0\leq k\leq n-1$. By Theorem \ref{integralU6n-theorem5}, $\Cay(U_{6n},S)$ is integral. By Equations (\ref{integralU6n-equation2}) and (\ref{integralU6n-equation3}), the spectrum of $\Cay(U_{6n}, S)$ is $\{[-3]^{2n-1},[0]^{4n}, 6n-3\}$.
\end{proof}

\medskip

\begin{remark}
\em{
Note that $\Cay(U_{6n},S)$ in Corollary \ref{integralU6n-cor3} is a strongly regular graph with parameters $(6n,6n-3,6n-6,6n-3)$. This constructs an infinite family of integral strongly regular graphs.
}
\end{remark}


\section{Conclusion}
In this paper, we first give a necessary and sufficient condition for integrality of Cayley graphs $\Cay(U_{6n}, S)$. And then we give another version expressed by the irreducible characters of cyclic groups. By this description, we get relationships between integrality of $\Cay(U_{6n}, S)$ and the Boolean algebra of cyclic groups. As applications, we construct some infinite families of integral Cayley graphs over $U_{6n}$. At the end of this paper, we would like to propose the following problem:
\begin{problem}
Characterize integral Cayley graphs over other non-abelian groups, such as metacyclic groups and nilpotent groups.
\end{problem}
%
%
%
%


\end{document}